\documentclass[11pt]{amsart}
\usepackage[english]{babel}
\usepackage{graphicx}
\newtheorem{theorem}{Theorem}[section]

\newtheorem{lemma}[theorem]{Lemma}

\theoremstyle{definition}

\theoremstyle{remark}
\newtheorem{remark}{Remark}
\newtheorem{example}{Example}
\numberwithin{equation}{section}
\newcommand{\real}{\mathbb R}
\def\natu{\mathbb N}
\def\dis{\displaystyle}
\def\sobol{H^{1}(0,T)}
\def\sobopern{(H^1_T (0,T))^n}
\def\intl{\int_{0}^{T}}
\def\periodica{L_{T}(\real,\real)}

\def\betaper{\beta_p ^{per}}
\def\betaant{\beta_p ^{ant}}

\def\afirm#1#2#3{\skipaline\noindent{\bf #1} \hfill
\begin{tabular}{p{#2}}{\sl #3} \end{tabular}\hfill $ $\skipaline}
\def\skipaline{\vskip12pt plus 1pt}

\begin{document}

\title[Periodic conservative systems]
{Stability, Resonance and Lyapunov Inequalities for Periodic
Conservative Systems}
\author{Antonio Ca\~{n}ada}
\author{Salvador Villegas}
\thanks{The authors have been supported by the Ministry of Education and
Science of Spain (MTM2008.00988)}
\address{Departamento de An\'{a}lisis
Matem\'{a}tico, Universidad de Granada, 18071 Granada, Spain.}
\email{acanada@ugr.es, svillega@ugr.es}


\begin{abstract}
This paper is devoted to the study of Lyapunov type inequalities
for periodic conservative systems. The main results are derived
from a previous analysis which relates  the best Lyapunov
constants to some especial (constrained or unconstrained)
minimization problems. We provide some new results on the
existence and uniqueness of solutions of nonlinear resonant and
periodic systems. Finally, we present some new conditions which
guarantee the stable boundedness of linear periodic conservative
systems.
\end{abstract}

\maketitle

\section{Introduction}

\noindent The classical Lyapunov criterion on the stability of
Hill's equation
\begin{equation}\label{hill}
u''(t) + q(t)u(t) = 0,\ t \in \real,
\end{equation}
with $q(\cdot)$ a $T-$periodic function, says that if
\begin{equation}\label{hill2}
q \in L^1 (0,T), \ \intl q(t) \ dt > 0, \ \intl q^+ (t) \ dt \leq
\frac{4}{T},
\end{equation}
then (\ref{hill}) is stable (in the sense of Lyapunov, i.e. any
solution $u(\cdot)$ of (\ref{hill}) satisfies $\sup_{t \in \real}\
( \vert u(t) \vert + \vert \dot{u}(t) \vert ) < \infty$)
(\cite{magnus}). Here $q^+ (t) = \max \{ q(t),0 \}$ denotes the
positive part of the function $q.$

Condition (\ref{hill2}) has been generalized in several ways. In
particular, in \cite{zhangli}, the authors provide optimal
stability criteria by using $L^{p}$ norms of $q^+,$ $1 \leq p \leq
\infty.$ In the proof, one of the main ideas is a useful relation
between the eigenvalues of (\ref{hill}) associated to periodic and
antiperiodic boundary conditions and those associated to Dirichlet
boundary conditions (see Theorem 4.3 in \cite{zhanglondon}).

Despite its undoubted interest, there are not many studies on the
stability properties for systems of equations
\begin{equation}\label{2ieq35}
u''(t) + Q(t)u(t)= 0,\ t \in \real,
\end{equation}
where the matrix function $Q(\cdot)$ is $T-$periodic. A notable
contribution was provided by Krein  in \cite{krein}. In this work,
the author assumes that $Q(\cdot) \in \Lambda,$ where $\Lambda$ is
defined as

\afirm{[$\Lambda$]}{10cm}{The set of real $n\times n$ symmetric
matrix valued function $Q(\cdot)$, with continuous and
$T-$periodic element functions $q_{ij}(t), \ 1 \leq i,j \leq n,$
such that (\ref{2ieq35}) has not nontrivial constant solutions and
$$ \intl \langle Q(t)k,k \rangle  \ dt \geq 0, \
\forall \ k \in \real^n. $$ }

\noindent Krein proved that in this case, all solutions of the
system (\ref{2ieq35}) are stably bounded (see Section 4 for the
precise definition of this property) if $\lambda_1
> 1,$ where $\lambda_1$ is the smallest positive eigenvalue of the
eigenvalue problem
\begin{equation}\label{2eq37}
u''(t) + \mu Q(t)u(t)= 0, \ t \in \real, \ u(0)+u(T) = u'(0)+u'(T)
= 0.
\end{equation}
The Lyapunov conditions (\ref{hill2}) and those given in
\cite{zhangli} for scalar equations, imply $\lambda_1 > 1,$ but
for systems of equations, and assuming $Q(\cdot) \in \Lambda,$ it
is not easy to give sufficient conditions to ensure the property
$\lambda_1 > 1.$ In fact, to the best of our knowledge, we do not
know a similar result to Theorem 4.3 in \cite{zhanglondon}, for
the case of systems of equations.

In \cite{clhi} the authors establish sufficient conditions for
having $\lambda_1
> 1$ which involve $L^1$ restrictions on the spectral radius of
some appropriate matrices which are calculated by using the matrix
$Q(t).$ It is easy to check that, even in the scalar case, these
conditions are independent from classical $L^1-$ Lyapunov
criterion (\ref{hill2}).

In Section 4 we present some new conditions which allow to prove
that $\lambda_1 > 1.$ These conditions are given in terms of the
$L^p$ norm of appropriate functions $b_{ii}(t), \ 1 \leq i \leq
n,$ related to (\ref{2ieq35}) through the inequality $Q(t) \leq
B(t), \ \forall\ t \in [0,T],$ where $B(t)$ is a diagonal matrix
with entries given by $b_{ii}(t), \ 1 \leq i \leq n.$ These
sufficient conditions are optimal in the sense explained in Remark
\ref{remark3} below. Here, the relation $C \leq D$ between $n
\times n$ symmetric matrices means that $D-C$ is positive
semi-definite.

Our main result in Section 4 is derived from a fundamental
relation between the best $L^p$ Lyapunov constant and the minimum
of some especial constrained minimization problems. This relation
is proved in Section 2. Think that the conditions (\ref{hill2})
are resonant conditions, in the sense that the real number zero is
the first eigenvalue of the periodic eigenvalue problem
\begin{equation}\label{1805101}
u''(t) + \mu u(t) = 0, \ u(0)-u(T) = u'(0)-u'(T) = 0
\end{equation}
so that, these constrained minimization problems arises in a
natural way. For other boundary conditions such as Dirichlet or
antiperiodic boundary ones, the minimization problems associated
to best Lyapunov constants are unconstrained minimization problems
(see \cite{talenti}, \cite{zhang}). Motivated by a completely
different problem (an isoperimetric inequality known as Wulff
theorem, of interest in crystallography), the authors studied in
\cite{dacoro1} (see also \cite{dacoro2}) similar variational
problems but, in our opinion, the relation between these
minimization problems and $L^p$ Lyapunov constants for periodic
boundary conditions, is established for the first time in Section
2 of the present paper (see \cite{camovimia}, for the case of
Neumann boundary conditions).

Another important application of Lyapunov inequalities is the
study of nonlinear resonant problems. If $G: \real^n \rightarrow
\real$ is a $C^2-$mapping and $A$ and $B$ are real symmetric $n
\times n$ matrices with respective eigenvalues $a_1 \leq \ldots
\leq a_n$ and $b_1 \leq \ldots \leq b_n$ satisfying
\begin{equation}\label{eq21}
\begin{array}{c}
A \leq G''(u) \leq B, \  \forall \ u \in \real^n,
\\ \\ 0 < a_i \leq b_i < \frac{4\pi^2}{T^2}, \ 1 \leq i \leq n,
\end{array}
\end{equation}
then, for each continuous and $T-$periodic function $h: \real
\rightarrow \real^n$, the periodic problem
\begin{equation}\label{eq20}
u''(t) + G'(u(t)) = h(t), \ t \in \ (0,T), \ u(0)-u(T) =
u'(0)-u'(T) = 0,
\end{equation}
has a unique solution (see \cite{ahmad}, \cite{brown} and
\cite{lazer}). This last result is also true by using more general
restrictions than (\ref{eq21}) which involves higher eigenvalues
of (\ref{1805101}) (think that $0$ and $\frac{4\pi^2}{T^2}$ are
the first two eigenvalues of the eigenvalue problem
(\ref{1805101})). The mentioned results only allow a weak
interaction between the nonlinear term $G'(u)$ and the spectrum of
the linear part (\ref{1805101}) in the following sense: by using
the variational characterization of the eigenvalues of a real
symmetric matrix, it may be easily deduced that (\ref{eq21}) imply
that the eigenvalues $g_1 (u) \leq \cdots \leq g_n (u)$ of the
matrix $G''(u),$ satisfy
\begin{equation}\label{eq22}
0 < a_i \leq g_i (u) \leq b_i < \frac{4\pi^2}{T^2},  \ \forall \ u
\in \real^n, \ 1 \leq i \leq n,
\end{equation}
and consequently, (\ref{eq21}), which is  a $L^\infty$
restriction, may be seen as a nonresonant hypothesis. In Section 3
we provide for each $p,$ with $1 \leq p \leq \infty,$ $L^p$
restrictions for boundary value problem (\ref{eq20}) to have a
unique solution. These are optimal in the sense shown in Remark
\ref{r2} below. They are given in terms of the $L^p$ norm of
appropriate functions $b_{ii}(t), \ 1 \leq i \leq n,$ related to
(\ref{eq20}) through the inequality $A(t) \leq G''(u) \leq B(t), \
\forall\ t \in [0,T],$ where $B(t)$ is a diagonal matrix with
entries given by $b_{ii}(t), \ 1 \leq i \leq n$ and $A(t)$ is a
convenient symmetric matrix which belongs to $\Lambda  $ and
consequently, this avoids the resonance at the eigenvalue $0.$
Since our conditions are given in terms of $L^p$ norms, we allow
to  the functions $g_i (u)$ to cross an arbitrary number of
eigenvalues as long as certain $L^p$ norms are controlled.

\section{Preliminary results on scalar Lyapunov inequalities and minimization problems}

\noindent This section will be concerned with some preliminary
results on Lyapunov inequalities for the periodic boundary value
problem
\begin{equation}\label{eq1}
u''(t) + a(t)u(t) = 0, \ t \in (0,T), \ u(0)-u(T) = u'(0) - u'(T)
= 0,
\end{equation}
and the antiperiodic boundary value problem
\begin{equation}\label{eq2}
u''(t) + a(t)u(t) = 0, \ t \in (0,T), \ u(0)+u(T) = u'(0) + u'(T)
= 0,
\end{equation}
where, from now on, we assume that $a \in L_{T} (\real,\real),$
the set of $T$-periodic functions $a: \real \rightarrow \real$
such that $a|_{[0,T]} \in L^1 (0,T).$

If we define the sets

\begin{equation}\label{eq3}
\Lambda^{per} = \{ a \in \periodica\setminus \{0\}: \dis \intl
a(t) \ dt \geq 0 \ \mbox{and} \ (\ref{eq1}) \ \mbox{has nontrivial
solutions} \ \}
\end{equation}
\begin{equation}\label{eq4}
\Lambda^{ant} = \{ a \in \periodica : \ (\ref{eq2}) \ \mbox{has
nontrivial solutions} \ \},
\end{equation}
let us observe that the positive eigenvalues of the eigenvalue
problem
\begin{equation}\label{eq7}
 u''(t) + \lambda u(t) = 0, \ t \in (0,T), \ u(0)-u(T) = u'(0)- u'(T) = 0, \end{equation}%
and the eigenvalues of the eigenvalue problem
\begin{equation}\label{eq8}
 u''(t) + \lambda u(t) = 0, \ t \in (0,T), \ u(0)+u(T) = u'(0)+ u'(T) = 0, \end{equation}%
belong, respectively,  to $\Lambda^{per}$ and $\Lambda^{ant}$.
Therefore, for each $p$ with $1 \leq p \leq \infty,$ we can
define, respectively, the $L^p$-Lyapunov constants $\beta_p
^{per}$ and $\beta_p ^{ant}$ for the periodic and the antiperiodic
problem, as the real numbers
\begin{equation}\label{eq5}
\begin{array}{c}
\beta_{p}^{per} \equiv \displaystyle \inf_{a \in
\Lambda^{per}\bigcap L^p (0,T)} \ \ \Vert a^+ \Vert_p, \ \
\beta_{p}^{ant} \equiv \displaystyle \inf_{a \in
\Lambda^{ant}\bigcap L^p (0,T)} \ \ \Vert a^+ \Vert_p,
\end{array}
\end{equation}
where
\begin{equation}\label{eq6}
\begin{array}{c}
\Vert a^+ \Vert_{p} = \left ( \displaystyle \int_{0}^{T} \vert a^+
(t)\vert ^{p} \ dt \right ) ^{1/p}, 1 \leq p < \infty; \ \ \Vert
a^+ \Vert_{\infty} =  sup \ ess \ \ a^+.
\end{array}
\end{equation}
An explicit expression for the constants $\betaper$ and $\betaant
$, as a function of $p$ and $T,$ has been obtained in \cite{zhang}
(see also \cite{camovimia}, \cite{cavidcds} and \cite{talenti} for
the case of Neumann, mixed and Dirichlet boundary conditions,
respectively).

A key point to extend the mentioned previous results on scalar
problems to systems of equations, is the characterization of the
$L^p$-Lyapunov constant as a minimum of a convenient minimization
scalar problem, where only some appropriate subsets of the space
$\sobol$ are used (here $\sobol$ is the usual Sobolev space). For
the Dirichlet problem this was done by Talenti (\cite{talenti})
and for the Neumann problem this was done by the authors in
\cite{camovimia}. In the next two lemmas, we treat, respectively,
with the periodic and the antiperiodic problem. In the proof, only
those innovative details are shown.

\begin{lemma}\label{l1}
If $1 \leq p \leq \infty$ is a given number, let us define the
sets $X_p ^{per}$ and the functionals $I_p ^{per}:X_p
^{per}\setminus \{ 0 \} \rightarrow \real $ as
\begin{equation}\label{eq10}
\begin{array}{c}
X_{1}^{per} = \{ v \in \sobol : v(0)- v(T) = 0, \displaystyle
\max_{t \in [0,T]}
v(t) + \displaystyle \min_{t \in [0,T]} v(t) = 0 \}, \\ \\
X_{p}^{per} = \left \{ v \in \sobol: v(0)- v(T)= 0, \dis \intl
\vert v \vert ^{\frac{2}{p-1}} v = 0 \right \},\ \mbox{if} \ \ 1 <
p < \infty,\\ \\
X_{\infty}^{per}  = \{ v \in \sobol: \ v(0)- v(T) = 0, \ \dis
\intl v = 0 \}, \\ \\
I_1 ^{per}(v) = \displaystyle \frac{\dis \intl v'^{2}}{\Vert v
\Vert_{\infty}^{2}}, \ \ I_{p}^{per}(v) =  \dis \frac{\dis \intl
v'^{2}}{\left ( \dis \intl \vert v \vert ^{\frac{2p}{p-1}} \right
)^{\frac{p-1}{p}}}, \ \mbox{if} \ \ 1 < p < \infty, \ \
I_{\infty}^{per}(v) = \displaystyle \frac{\dis \intl
v'^{2}}{\displaystyle \intl v^{2}}
\end{array}
\end{equation}
Then, the $L_p$ Lyapunov constant $\beta_p ^{per}$ defined in
(\ref{eq5}), satisfies
\begin{equation}\label{eq11} \begin{array}{c}
\beta_{p}^{per} =   \dis \min_{X_p ^{per} \setminus \{ 0 \}} \
I_{p} ^{per}, \ \ 1 \leq p \leq \infty.
\end{array}
\end{equation}

\end{lemma}

\begin{proof}

$ $

$ $

{\bf The case $p=1.$} It is very well known that $\beta_1 ^{per} =
\frac{16}{T}$ (\cite{huai},\cite{zhang}). Now, if $u \in X_1
^{per}\setminus \{ 0 \},$ then there exists $x_0 \in [0,T]$ such
that $u(x_0) = 0.$ Taking into account that $u$ can be extended as
a $T-$ periodic function to $\real,$ if we define the function
$v(x) = u(x+x_0), \ \forall \ x \in \real,$ then $v|_{[0,T]}  \in
X_1^{per} \setminus \{ 0 \},$ $v(0) = v(T) = 0$ and $I_1^{per}(u)
= I_1^{per}(v).$ In addition (if it is necessary, we can choose
$-v$ instead of $v$), there exist $0 < x_1 < x_2 < x_3 < T$ such
that
$$ v(x_1) = \max_{[0,T]} \ v, \ \ v(x_2) = 0, \ \ v(x_3) =
\min_{[0,T]} \ v.
$$
If $x_0 = 0, x_4 = T,$ it follows from the Cauchy-Schwarz
inequality
\begin{equation}\label{050510}
\begin{array}{c}
\intl v'^{2} = \dis \sum_{i=0}^{3} \ \int_{x_i}^{x_{i+1}} v'^{2}
\geq \dis \sum_{i=0}^{3} \dis \frac{\left ( \int_{x_i}^{x_{i+1}}
\vert v' \vert \right ) ^2 }{x_{i+1} -x_i} \geq \\ \\
\dis \sum_{i=0}^{3} \dis \frac{\left ( \int_{x_i}^{x_{i+1}}  v'
\right ) ^2 }{x_{i+1} -x_i}= \dis \sum_{i=0}^3 \dis
\frac{(v(x_{i+1}) - v(x_i))^2 }{x_{i+1}-x_i} = \\ \\ \Vert v
\Vert_{\infty} ^2 \dis \sum_{i=0}^3 \dis \frac{1}{x_{i+1}-x_i}
\geq \frac{16}{T} \Vert v \Vert_{\infty}^2
\end{array}
\end{equation}
Consequently
\begin{equation}\label{eq13}
I_{1}^{per} (u) = \dis \frac{\intl u'^{2}}{\Vert u
\Vert_{\infty}^2} = I_{1}^{per}(v) \geq \dis \frac{16}{T}, \
\forall \ u \in X_1 ^{per}\setminus \{ 0 \}.
\end{equation}
On the other hand, the function $w \in X_1 ^{per}\setminus \{ 0
\}$ defined as
\begin{equation}\label{eq14}
w(x) = \left \{
\begin{array}{l}
x, \ \mbox{if} \ 0 \leq x \leq T/4, \\
-(x-\frac{T}{2}), \ \mbox{if} \ T/4 \leq x \leq 3T/4, \\
(x-T), \ \mbox{if} \ 3T/4 \leq x \leq T, \\
\end{array}
\right.
\end{equation}
satisfies
$$
\dis \frac{\intl w'^{2}}{\Vert w \Vert_{\infty}^2} = \dis
\frac{16}{T}
$$
Consequently, the case $p=1$ is proved.

{\bf The case $p=\infty.$} It is very well known that
$\beta_\infty ^{per} = \frac{4\pi^2}{T^2}$, the first positive
eigenvalue of the eigenvalue problem (\ref{eq7}) (see
\cite{zhang}). From its variational characterization, we obtain
$$
\beta_{\infty}^{per} = \min_{X_\infty ^{per} \setminus \{ 0 \}} \
I_{\infty} ^{per}.
$$

{\bf The case $1 < p < \infty$.} Let us denote
$$m_{p} = \inf_{X_p ^{per}\setminus \{ 0 \}}\ I_p ^{per}.$$ If $\{ u_{n} \} \subset
X_{p}^{per} \setminus \{ 0 \}$ is a minimizing sequence, since the
sequence $\{ k_{n}u_{n} \}, \ k_{n} \neq 0, $ is also a minimizing
sequence, we can assume without loos of generality that $\dis
\intl \vert u_{n} \vert ^{\frac{2p}{p-1}} = 1.$ Then $\left \{
\dis \intl \vert u_{n}'^{2} \vert \right \} $ is also bounded.
Moreover, for each $u_{n}$ there is $x_{n} \in (0,T)$ such that
$u_{n}(x_{n}) = 0.$ Therefore, $\{ u_{n} \}$ is bounded in
$\sobol.$ So, we can suppose, up to a subsequence, that $u_{n}
\rightharpoonup u_{0}$ in $\sobol$ and $u_{n} \rightarrow u_{0}$
in $C[0,L]$ (with the uniform norm). The strong convergence in
$C[0,L]$ gives us $u_0 (0)-u_0 (T) = 0$ and $\dis \intl \vert
u_{0} \vert ^{\frac{2p}{p-1}} = 1$. Therefore, $u_{0} \in
X_{p}^{per} \setminus \{ 0 \}.$ The weak convergence in $\sobol$
implies $I^{per}_{p}(u_{0}) \leq \liminf \ I^{per}_{p}(u_{n}) =
m_{p}.$ Then $u_{0}$ is a minimizer of $I^{per}_p$ in $X^{per}_p
\setminus \{ 0 \}.$
\newline
Since $X^{per}_{p} = \{ \ u \in \sobol: u(0)-u(t) = 0, \ \varphi
(u) = 0 \},$ $\varphi (u) = \dis \intl \vert u \vert
^{\frac{2}{p-1}} u ,$ if $u_{0} \in X^{per}_{p} \setminus \{ 0 \}$
is any minimizer of $I^{per}_{p}$, Lagrange multiplier Theorem
implies that there is $\lambda \in \real$ such that
$$
H'(u_{0})(v) + \lambda \varphi '(u_{0})(v) = 0, \ \forall v \in
\sobol \ \mbox{such that} \ v(0)-v(T) = 0.
$$
Here $H: \sobol \rightarrow \real$ is defined by
$$
H(u) = \dis \intl u'^{2} - m_{p} \left ( \dis \intl \vert u \vert
^{\frac{2p}{p-1}} \right )^{\frac{p-1}{p}}
$$
Also, since $u_{0} \in X_{p}^{per}$ we have $H'(u_{0})(1) = 0.$
Moreover $H'(u_{0})(v) = 0, \ \forall \ v \in \sobol : \ v(0)
-v(t) = 0, \ \varphi'(u_{0})(v) = 0.$ Finally, as any $v \in
\sobol$ satisfying $v(0)-v(T) = 0,$ may be written in the form $v
= \alpha + z, \ \alpha \in \real,$ and $z\in \sobol$ satisfying
$z(0)-z(T) = 0, \ \varphi'(u_{0})(z) = 0,$ we conclude
$H'(u_{0})(v) = 0, \ \forall \ v \in \sobol,$  such that
$v(0)-v(T)= 0.$ This implies that $u_0$ satisfies the problem
\begin{equation}\label{eq15}
\begin{array}{c}
u_0''(x) + A_{p}(u_0)\vert u_0 (x)\vert^{\frac{2}{p-1}} u_0 (x) =
0, \ x \in (0,T), \\  u_0 (0)-u_0 (T) = 0, \ u_0 '(0)-u_0 ' (T) =
0,
\end{array}
\end{equation}
where
\begin{equation}\label{eq16}
A_{p}(u_0) = m_{p} \left ( \dis \intl \vert u_0\vert
^{\frac{2p}{p-1}} \right )^{\frac{-1}{p}}
\end{equation}

If one has an exact knowledge about the number and distribution of
the zeros of the functions $u_0$ and $u_0 ',$ the Euler equation
(\ref{eq15}) can be integrated (see \cite{camovimia}, Lemma 2.7).
In our case, it is not restrictive to assume $u_0 (0) = u_0 (T) =
0$ (see the previous case $p=1$). Then, if we denote the zeros of
$u_0$ in $[0,T]$ by $0 = x_0 < x_2 < \ldots < x_{2n} = T$ and the
zeros of $u_0 '$ in $(0,T)$ by $ x_1 < x_3 < \ldots < x_{2n-1}$,
we obtain
\begin{equation}\label{eq17}
m_{p} = \dis \frac{4n^{2}I^{2}p}{T^{2-\frac{1}{p}}
(p-1)^{1-\frac{1}{p}} (2p-1)^{1/p}}, \end{equation} where $ \ I =
\dis \int_{0}^{1} \frac{ds}{\left ( 1 - s^{\frac{2p}{p-1}} \right
) ^{1/2}}.$

The novelty here is that, for the periodic boundary value problem
(\ref{eq15}), $n \geq 2$ (see, again, the previous case $p=1$),
while for the Neumann and Dirichlet problem $n \geq 1.$

The conclusion is that
\begin{equation}\label{eq18}
m_{p} = \dis \frac{16I^{2}p}{T^{2-\frac{1}{p}}
(p-1)^{1-\frac{1}{p}} (2p-1)^{1/p}}, \end{equation} that is, four
times the corresponding $L^p-$Lyapunov constant for the Dirichlet
and the Neumann problem. Finally, in \cite{zhang} it is shown that
this is, exactly, the $L^p-$Lyapunov constant for the periodic
problem. Consequently, $m_p = \beta_{p}^{per}, \ 1 < p < \infty.$
\end{proof}

\begin{remark}\label{2105101}
It is easily deduced from the previous discussion that the set
$\Lambda^{per}$ in (\ref{eq3}) can be replaced by $$ \Lambda^{per}
= \{ a \in \periodica: \dis \intl a(t) \ dt >0 \ \mbox{and} \
(\ref{eq1}) \ \mbox{has nontrivial solutions} \ \}
$$

Also, if $u \in X_1 ^{per} \setminus \{ 0 \}$ is such that $I_1
^{per} (u) = \frac{16}{T},$ then all the inequalities of
(\ref{050510}) transforms into equalities. In particular $x_{i+1}
- x_i = \frac{T}{4}, \ 0 \leq i \leq 3,$ and, again, the
Cauchy-Schwartz inequality (equality in this case) implies that
the function $v'$ is constant in each interval $[x_i,x_{i+1}], 0
\leq i \leq 3.$ We deduce that there exists a nontrivial constant
$c$ and $x_0 \in [0,T]$ such that $u(\cdot) = cw(\cdot + x_0),$
where $w$ is given in (\ref{eq14}).
\end{remark}

\begin{remark}\label{r1}
Motivated by a completely different problem (an isoperimetric
inequality known as Wulff theorem, of interest in
crystallography), the authors studied in \cite{dacoro1} similar
variational problems (see also \cite{dacoro2} for more general
minimization problems). In our opinion, these variational problems
are not related with Lyapunov inequalities in \cite{dacoro1}. To
the best of our knowledge, this was shown in \cite{talenti} for
Dirichlet boundary conditions and in \cite{camovimia} for Neumann
boundary conditions. Since $0$ is the first eigenvalue for Neumann
and periodic boundary conditions, it is necessary to impose an
additional restriction to the definition of the spaces $X_p, \ 1
\leq p \leq \infty$ in the case of Neumann and periodic
conditions. This is not necessary in the case of Dirichlet or
antiperiodic boundary ones (see the next lemma).
\end{remark}

\begin{lemma}\label{l2}
If $1 \leq p \leq \infty$ is a given number, let us define the
sets $X_p ^{ant}$ and the functional $I_p ^{ant}: X_p
^{ant}\setminus \{ 0 \} \rightarrow \real,$ as
\begin{equation}\label{eqant10}
\begin{array}{c}
X_p ^{ant} = \left \{ v \in \sobol: v(0)+ v(T)= 0 \right \}, \ 1
\leq p \leq \infty, \\ \\
I_1 ^{ant}(v) = \displaystyle \frac{\dis \intl v'^{2}}{\Vert v
\Vert_{\infty}^{2}}, I_p ^{ant}(v) = \dis \frac{\dis \intl
v'^{2}}{\left ( \dis \intl \vert v \vert ^{\frac{2p}{p-1}} \right
)^{\frac{p-1}{p}}}, \ \mbox{if} \ \ 1 < p < \infty, \ \ I_{\infty}
^{ant} (v) = \displaystyle \frac{\dis \intl v'^{2}}{\displaystyle
\intl v^{2}}
\end{array}
\end{equation}
Then, the $L_p$ Lyapunov constant $\beta_p ^{ant}$ defined in
(\ref{eq5}), satisfies
\begin{equation}\label{eqant11} \begin{array}{c}
\beta_{p}^{ant} = \displaystyle \min_{X_p ^{ant} \setminus \{ 0
\}} \ I_{p} ^{ant},\ \ 1 \leq p \leq \infty.
\end{array}
\end{equation}
\end{lemma}
\begin{proof}
{\bf The case $p=1.$} It is very well known that $\beta_1 ^{ant} =
\frac{4}{T}$ (\cite{zhang}). Now if $u \in X_1 ^{ant} \setminus \{
0 \},$ let us define the function $\tilde{u}: [0,2T] \rightarrow
\real,$ as
$$
\tilde{u}(x) = \left \{ \begin{array}{l} u(x), \ \mbox{if}\ 0 \leq
x \leq T, \\
-u(x-T), \ \mbox{if}\ T \leq x \leq 2T.
\end{array}
\right.
$$
It is easily checked that $\tilde{u} \in H^{1} (0,2T) \setminus \{
0 \}, \ \tilde{u}(0) = \tilde{u}(2T), \ \max_{[0,2T]} \ \tilde{u}
+\min_{[0,2T]} \ \tilde{u} = 0.$ It is deduced, from the first
part of Lemma \ref{l1}, that
$$
\dis \frac{2\int_0^{T} u'^2}{\Vert u\Vert^2 _{L^{\infty}(0,T)}} =
\dis \frac{\int_0^{2T} \tilde{u}'^2}{\Vert \tilde{u}\Vert^2
_{L^{\infty}(0,2T)}} \geq \dis \frac{16}{2T}.
$$
and consequently,
$$
 \dis \frac{\int_0^{T} u'^2}{\Vert
u\Vert^2 _{L^{\infty}(0,T)}} \geq \dis \frac{4}{T}, \ \forall \ u
\in X_1 ^{ant} \setminus \{ 0 \}.
$$
Also, the function $v: [0,T] \rightarrow \real,$ defined as
\begin{equation}\label{1105101}
v(x) = \dis \frac{T}{2} - x,
\end{equation} belongs to $X_1
^{ant} \setminus \{ 0 \}$ and $ \dis \frac{\int_0^{T}
(v')^2}{\Vert v\Vert^2 _{L^{\infty}(0,T)}} = \dis \frac{4}{T}.$ As
a consequence, $\dis \min_{X_1 ^{ant} \setminus \{ 0 \}} \ I_1
^{ant} = \dis \frac{4}{T}.$

{\bf The case $p=\infty.$} It is very well known that
$\beta_\infty ^{ant} = \frac{\pi^2}{T^2},$ the first eigenvalue of
the eigenvalue problem (\ref{eq8}) (see \cite{zhang}). From its
variational characterization, we obtain
$$
\beta_{\infty}^{ant} = \min_{X_\infty ^{ant} \setminus \{ 0 \}} \
I_{\infty} ^{ant}.
$$
{\bf The case $1 < p < \infty$.} Let us denote
$$M_{p} = \inf_{X_p ^{ant}\setminus \{ 0 \}}\ I_p ^{ant}.$$ If $\{ u_{n} \} \subset
X_{p}^{ant} \setminus \{ 0 \}$ is a minimizing sequence, since the
sequence $\{ k_{n}u_{n} \}, \ k_{n} \neq 0, $ is also a minimizing
sequence, we can assume without loos of generality that $\dis
\intl \vert u_{n} \vert ^{\frac{2p}{p-1}} = 1.$ Then $\left \{
\dis \intl \vert u_{n}'^{2} \vert \right \} $ is also bounded.
Moreover, for each $u_{n}$ there is $x_{n} \in [0,T]$ such that
$u_{n}(x_{n}) = 0.$ Therefore, $\{ u_{n} \}$ is bounded in
$\sobol.$ So, we can suppose, up to a subsequence, that $u_{n}
\rightharpoonup u_{0}$ in $\sobol$ and $u_{n} \rightarrow u_{0}$
in $C[0,L]$ (with the uniform norm). The strong convergence in
$C[0,L]$ gives us $u_0 (0)+u_0 (T) = 0$. Therefore, $u_{0} \in
X_{p}^{ant} \setminus \{ 0 \}.$ The weak convergence in $\sobol$
implies $I^{ant}_{p}(u_{0}) \leq \liminf \ I^{ant}_{p}(u_{n}) =
M_{p}.$ Then $u_{0}$ is a minimizer. Therefore,
$$
H'(u_{0})(v) = 0, \ \forall v \in \sobol \ \mbox{such that} \
v(0)+v(T) = 0.
$$
Here $H: \sobol \rightarrow \real$ is defined by
$$
H(u) = \dis \intl u'^{2} - M_{p} \left ( \dis \intl \vert u \vert
^{\frac{2p}{p-1}} \right )^{\frac{p-1}{p}}
$$
This implies that $u_0$ satisfies the problem
\begin{equation}\label{eqant13}
\begin{array}{c}
u_0''(x) + A_{p}(u_0)\vert u_0 (x)\vert^{\frac{2}{p-1}}u_{0}(x) =
0, \ x \in (0,T), \\  u_0 (0)+u_0 (T) = 0, \ u_0 '(0)+u_0 ' (T) =
0,
\end{array}
\end{equation}
where
\begin{equation}\label{eqant14}
A_{p}(u_0) = M_{p} \left ( \dis \intl \vert u_0\vert
^{\frac{2p}{p-1}} \right )^{\frac{-1}{p}}
\end{equation}
Since the function $a(x) \equiv A_{p}(u_0)\vert u_0
(x)\vert^{\frac{2}{p-1}}$ satisfies $a(0) = a(T),$ it is not
restrictive to assume that, additionally, $u_0 (0) = u_0 (T) = 0.$
Then, we deduce from Lemma 2.7 in \cite{camovimia} that
\begin{equation}\label{eqant17}
M_{p} = \dis \frac{4n^{2}I^{2}p}{T^{2-\frac{1}{p}}
(p-1)^{1-\frac{1}{p}} (2p-1)^{1/p}}, \end{equation} where $ \ I =
\dis \int_{0}^{1} \frac{ds}{\left ( 1 - s^{\frac{2p}{p-1}} \right
) ^{1/2}}$ and $n \in \natu$ is such that we denote the zeros of
$u_0$ in $[0,T]$ by $0 = x_0 < x_2 < \ldots < x_{2n} = T$ and the
zeros of $u_0 '$ in $(0,T)$ by $ x_1 < x_3 < \ldots < x_{2n-1}.$
The novelty here, with respect to the periodic boundary value
problem, is that $n \geq 1.$ The conclusion is that
\begin{equation}\label{eqant18}
M_{p} = \dis \frac{m_p}{4}\ \ \mbox{if} \ 1 < p < \infty.
\end{equation}
Finally, it is known that $\beta_p ^{ant} = \frac{ \beta_p
^{per}}{4},$ if $1 < p < \infty$ (see \cite{zhang}). The Lemma is
proved.
\end{proof}

\begin{remark}\label{remark2405101}
We must remark that, if $w \in X_1 ^{ant} \setminus \{ 0 \}$ is
such that $I_1 ^{ant} (w) = \frac{4}{T},$ then  there exists a
nontrivial constant $c$ and $x_0 \in [0,T]$ such that $w(x) =
c(\frac{T}{2} - |x-x_0 |), \ \forall \ x \in [0,T].$
\end{remark}

\section{Resonant nonlinear periodic systems}

In this section we consider systems of equations of the type
(\ref{2eq17}) below, which models the Newtonian equation of motion
of a mechanical system subject to conservative internal forces and
periodic external forces. The main result is the following.

\begin{theorem}\label{t1}
Let $G : \real \times \real^n \rightarrow \real, \ (t,u)
\rightarrow G(t,u),$  be a continuous function, $T-$ periodic with
respect to the variable $t$ and satisfying:
\begin{enumerate}
\item $u \rightarrow G(t,u)$ is of class $C^2(\real^n,\real),$ for
every $t \in \real.$
\item %
There exist continuous and $T-$periodic matrix functions
$A(\cdot),$ $B(\cdot),$ with $A(t)$ symmetric, $B(t)$ diagonal
with entries $b_{ii}(t),$ and $p_i \in [1,\infty] \ 1 \leq i \leq
n,$ such that
\begin{equation}\label{2eq15}
\left.
\begin{array}{c}
A(t) \leq G_{uu}(t,u) \leq B(t),  \ \forall (t,u) \in
\real^{n+1},\\ \\ \intl \langle A(t)k,k \rangle \ dt
> 0, \ \forall \ k \in \mathbb{R}^n \setminus \{ 0 \},
\\ \\\Vert b_{ii}^+ \Vert_{p_i} < \beta_{p_i}^{per}, \ \mbox{if} \
p_i \in (1,\infty], \ \ \Vert b_{ii}^+ \Vert_{p_i} \leq
\beta_{p_i}^{per}, \ \mbox{if} \ p_i =1.
\end{array}
\right \}
\end{equation}
\end{enumerate}
Then the boundary value problem
\begin{equation}\label{2eq17}
u''(t)+ G_u (t,u(t)) = 0,  \ t\in \real,  \ u(0)-u(T) =
u'(0)-u'(T) = 0,
\end{equation}
has a unique solution.
\end{theorem}
\begin{proof}
It is based on two steps. In the first one we prove the uniqueness
property. This suggests the way to prove existence of solutions.

$ $

{\bf 1.- Uniqueness of solutions of (\ref{2eq17}).}

Let us denote by $H^1 _T (0,T)$ the subset of $T-$periodic
functions of the Sobolev space $\sobol.$ Then, if $v \in
\sobopern$ and $w \in \sobopern$ are two solutions of
(\ref{2eq17}),  the function $u = v-w$ is a solution of the
problem
\begin{equation}\label{2eq18}
u''(t) + C(t)u(t)= 0 , \ t \in \real, \ u(0)-u(T) = u'(0)-u'(T) =
0,
\end{equation}
where  $C(t) = \dis \int_{0}^{1} G_{uu}(t,w(t)+\theta u(t)) \ d
\theta$ (see \cite{lang}, p. 103, for the mean value theorem for
the vectorial function $G_u (t,u)$). In addition,
\begin{equation}\label{eq25}
A(t)\leq C(t)\leq B(t), \ \forall \ t \in \real
\end{equation}
and
$$
\displaystyle \int_{0}^{T} \langle u'(t),z'(t) \rangle \  =
\displaystyle \int_{0}^{T} \langle C(t)u(t),z(t)\rangle , \
\forall \ z \in \sobopern.
$$
In particular, we have
\begin{equation}\label{2eq3}
\begin{array}{c}
\displaystyle \intl \langle u'(t),u'(t)\rangle \   = \displaystyle
\intl \langle C(t)u(t),u(t)\rangle  , \ \\ \dis \intl \langle
C(t)u(t),k\rangle \  = \dis \intl \langle C(t)k,u(t)\rangle \  =
0, \ \forall \ k \in \real^n .
\end{array}
\end{equation}
Therefore, for each $k \in \real^n, $ we have
$$
\begin{array}{c}
\dis \intl \langle (u(t)+k)',(u(t)+k)'\rangle \  = \dis \intl \langle u'(t),u'(t)\rangle \  = \\
\\
\displaystyle \intl \langle C(t)u(t),u(t)\rangle \   \leq
\displaystyle \intl \langle C(t)u(t),u(t)\rangle \  + \dis \intl \langle C(t)u(t),k\rangle \ + \\
\\  \dis \intl \langle C(t)k,u(t)\rangle \ + \dis \intl \langle C(t)k,k\rangle \  =\\ \\ \dis
\intl \langle C(t)(u(t)+k),u(t)+k\rangle \  \leq \dis \intl
\langle B(t)(u(t)+k),u(t)+k\rangle .
\end{array}
$$
If $u = (u_i), $ then for each $i, 1 \leq i \leq n,$ we choose
$k_i \in \real$ satisfying $u_i + k_i \in X_{p_i} ^{per},$ the set
defined in Lemma \ref{l1}. By using previous inequality, Lemma
\ref{l1} and H{\"o}lder inequality, we obtain
\begin{equation}\label{2eq8}
\begin{array}{c}
\dis \sum_{i=1}^n \beta_{p_i} ^{per} \Vert (u_{i} + k_i)^2
\Vert_{\frac{p_i}{p_i -1}} \leq \dis \sum_{i=1}^n \intl (u_i
(t)+k_i)'^2 \leq \\ \dis \sum_{i=1}^n \intl b_{ii}^+ (t)(u_i
(t)+k_i)^2 \leq \dis \sum_{i=1}^n \Vert b_{ii}^+ \Vert_{p_i} \Vert
(u_{i} + k_i)^2 \Vert_{\frac{p_i}{p_i -1}},
\end{array}
\end{equation}
where
$$
\begin{array}{c}
\frac{p_i}{p_i -1} = \infty, \ \ \mbox{if} \ p_i = 1 \\ \\
\frac{p_i}{p_i -1} = 1, \ \ \mbox{if} \  p_i = \infty. \end{array}
$$

Therefore we have
\begin{equation}\label{2eq9}
\dis \sum_{i=1}^{n} (\beta_{p_i}^{per} - \Vert b_{ii}^+
\Vert_{p_i}) \Vert (u_{i} + k_i)^2 \Vert_{\frac{p_i}{p_i -1}} \leq
0.
\end{equation}
Now, (\ref{2eq9}) implies that, necessarily, $u \equiv 0$ (and as
a consequence, $v \equiv w$). To see this, if $u$ is a nontrivial
function, then the function $u+k$ is also a nontrivial function.
In fact, if $u+k \equiv 0,$ we deduce that (\ref{2eq18}) has the
nontrivial and constant solution $-k$ which imply
$$
0 = \intl \langle C(t)k,k\rangle  \ dt \ \geq \intl \langle
A(t)k,k\rangle \ dt.
$$
This is a contradiction with (\ref{2eq15}).

Then, if $u+k$ is a nontrivial function, some component, say $u_j
+ k_j,$ is nontrivial.

If $p_j \in (1,\infty],$ then $(\beta_{p_j} ^{per} - \Vert
b_{jj}^+ \Vert_{p_j}) \Vert (u_{j} + k_j)^2 \Vert_{\frac{p_j}{p_j
-1}}$ is strictly positive and from (\ref{2eq15}), all the other
summands in (\ref{2eq9}) are nonnegative. This is a contradiction
with (\ref{2eq9}).

If $p_j =1,$ we must take into account that $\beta_1 ^{per}$ is
only attained in nontrivial functions of the form $y(t) =
cw(t+x_0),$ where $w(t)$ is given in (\ref{eq14}), $c$ is a
nontrivial constant and $x_0 \in [0,T]$ (see Remark \ref{2105101}
above). Any function $y$ of this type do not belong to $C^1
([0,T],\real).$ Since any solution $u\in \sobopern$ of
(\ref{2eq18}) belongs to $C^1 ([0,T],\real)$, we have
$$\beta_{p_j} \Vert (u_{j} + k_j)^2 \Vert_{\frac{p_j}{p_j -1}} <
\intl (u_j (x)+k_j)'^2.$$ This implies that the inequality
(\ref{2eq9}) is strict and this is a contradiction with
(\ref{2eq15}).

$ $

{\bf 2.- Existence of solutions of (\ref{2eq17}).}

First, we write (\ref{2eq17}) in the equivalent form
\begin{equation}\label{2eq19} \left.
\begin{array}{cc}
u''(t) + D(t,u(t))u(t) + G_u(t,0) = 0, \ t \in \real,  \\
 u(0)-u(T) = u'(0) - u'(T) = 0,
\end{array}\right\}
\end{equation}
where the  function $D: \real \times \real^n \rightarrow
\mathcal{M}(\real)$ is defined by $D(t,z) = \dis \int_{0}^{1}
G_{uu}(t,\theta z) \ d \theta.$ Here $\mathcal{M}(\real)$ denotes
the set of real $n \times n$ matrices.

If $C_{T} (\real,\real)$ is the set of real $T-$periodic and
continuos functions defined in $\real,$ let us denote $X=(C_T
(\real,\real))^n$ with the uniform norm (if  $y(\cdot) = (y^1
(\cdot),\cdots,y^n (\cdot)) \in X, $ then $\Vert y \Vert_X =
\displaystyle \sum_{k=1}^n \Vert y^k (\cdot) \Vert_{\infty}$).
Since
\begin{equation}\label{eq2306085}
A(t) \leq D(t,z) \leq B(t), \ \forall \ (t,z) \in \real \times
\real^n,
\end{equation}
we can define the operator $H: X \rightarrow X,$ by $Hy = u_{y}$,
being $u_{y}$ the unique solution of the linear problem
\begin{equation}\label{2q20}\left.
\begin{array}{cc}
u''(t)+ D(t,y(t))u(t) + G_u (t,0)= 0,  \ t \in  \real, \\
u(0)-u(T) = u'(0) - u'(T) = 0.
\end{array}\right\}
\end{equation}
In fact, (\ref{2q20}) is a nonhomogeneous linear problem such that
the corresponding homogeneous one has only the trivial solution
(as in the previous step on uniqueness).

We will show that $H$ is completely continuous and that $H(X)$ is
bounded. The Schauder's fixed point theorem provides a fixed point
for $H$ which is a solution of (\ref{2eq17}).

The fact that $H$ is completely continuous is a consequence of the
Arzela-Ascoli Theorem. It remains to prove that $H(X)$ is bounded.
Suppose, contrary to our claim, that $H(X)$ is not bounded. In
this case, there would exist a sequence $\{ y_{n} \} \subset X$
such that $\Vert u_{y_{n}}\Vert _{X} \rightarrow \infty.$ From
(\ref{eq2306085}), and passing to a subsequence if necessary, we
may assume that, for each $1 \leq i,j \leq n,$ the sequence of
functions $\{ D_{ij}(\cdot,y_{n}(\cdot))\}$ is weakly convergent
in $L^{p}(\Omega)$ to a function $E_{ij}(\cdot)$ and such that if
$E(t) = (E_{ij}(t)),$ then $A(t) \leq E(t) \leq B(t)$, a.e. in
$\real,$ (\cite{lema}, page 157).

If $z_{n} \equiv \dis \frac{u_{y_{n}}}{\Vert u_{y_{n}} \Vert
_{X}},$ passing to a subsequence if necessary, we may assume that
$z_{n} \rightarrow z_{0}$ strongly in $X,$  where $z_{0}$ is a
nonzero vectorial function satisfying
\begin{equation}\label{eq34}\left .
\begin{array}{cc}
 z_{0} ''(t)+ E(t)z_{0}(t)= 0, \ t \in \real \\
 z_0 (0)-z_0(T) = z_0 '(0) - z_0 '(T) = 0.
\end{array}\right \}
\end{equation}
But,  $A(t) \leq E(t) \leq B(t),\ \forall \ t \in \real $ and, as
in the first step on uniqueness, this implies  that the unique
solution of (\ref{eq34}) is the trivial one. This is a
contradiction with the fact that $\Vert z_0 \Vert _X = 1$ and, as
a consequence, $H(X)$ is bounded.
\end{proof}

\begin{remark}\label{r2}
Previous Theorem is optimal in the following sense: for any given
positive numbers $\gamma_i, \ 1 \leq i \leq n,$ such that at least
one of them, say $\gamma_j,$ satisfies
\begin{equation}\label{2optimalidad}
\gamma_j > \beta_{p_j} ^{per}, \ \mbox{for some} \ p_j \in
[1,\infty],
\end{equation} there exists a diagonal  $n\times n$ matrix
$A(\cdot)$ with continuous and $T-$periodic entries $a_{ii}(t), \
1 \leq i \leq n,$ satisfying $\Vert a_{ii}^+ \Vert_{p_i} <
\gamma_{i}, \ 1 \leq i \leq n,$ $ \intl \langle A(t)k,k \rangle \
dt
> 0, \ \forall \ k \in \mathbb{R}^n \setminus \{ 0 \}$ and a continuous and
$T-$periodic function $h: \real \rightarrow \real^n,$ such that
the boundary value problem
\begin{equation}\label{2s1}
u''(t) + A(t)u(t) = h(t), \ t \in (0,T), \ u(0)-u(T) = u'(0) -
u'(T)= 0,
\end{equation}
has not solution.

To see this, if $\gamma_j$ satisfies (\ref{2optimalidad}), then
there exists some continuous and $T-$periodic function $a(t), $
with $\intl a (t) \ dt >0, $ and $\Vert a^+ \Vert_{p_j} <
\gamma_j,$ such that the scalar problem
$$w''(t) + a(t)w(t) = 0, \ t \in (0,T), \ \ w(0)-w(T) = w'(0) - w'(T) = 0,
$$
has nontrivial solutions (see the definition of $\beta_{p_j}
^{per}$ in (\ref{eq5}) and Remark \ref{2105101}). If $w_j$ is one
of these nontrivial solutions, and we choose
$$
\begin{array}{c}
a_{jj}(t) = a(t), \ \ a_{ii}(t) = \delta \in \real^+, \ \mbox{if}
\  i\neq j,\\
h_j \equiv w_j, \ h_j \equiv 0,\  \mbox{if} \ i\neq j, \end{array}
$$
with $\delta$ sufficiently small, then (\ref{2s1}) has not
solution.
\end{remark}

\begin{example}\label{ejemplo1}
Now, we show an example which can not be studied by using the
results proved by Ahmad and Lazer in \cite{ahmad} and
\cite{lazer}, respectively, and Brown and Lin in \cite{brown}. In
fact, in the next example, we allow to the eigenvalues of the
matrix $G_{uu}(t,u)$ in the boundary value problem (\ref{2eq17}),
to cross an arbitrary number of eigenvalues of (\ref{eq7}).

To begin with the example, let $H : \real^n \rightarrow \real, \ u
\rightarrow H(u)$ be a given function of class
$C^2(\real^n,\real)$  such that
\begin{enumerate}
\item %
There exist a real constant symmetric $n\times n$ matrix $A$ and a
real constant diagonal matrix $B, $with respective eigenvalues
$$\begin{array}{c} a_1 \leq a_2 \leq \ldots \leq a_n, \\ b_1 \leq b_2 \leq \ldots
\leq b_n, \end{array}
$$ satisfying
\item $A \leq H_{uu}(u) \leq B, \ \forall \ u \in \real^n,$ \item
$0 < a_k \leq b_k < 1, \ 1 \leq k \leq n.$
\end{enumerate}

Then for each continuous and $2\pi-$ periodic function $h:\real
\rightarrow \real^n,$ the boundary value problem
\begin{equation}\label{22eq17}
u''(t)+ H_u (u(t)) = h(t),  \ t\in \real,  \ u(0)-u(2\pi) =
u'(0)-u'(2\pi) = 0,
\end{equation}
has a unique solution. In fact, this is a particular case of more
general results proved in \cite{ahmad} and \cite{lazer}  which
involve higher eigenvalues of the eigenvalue problem (\ref{eq7}).

$ $

Now, if  $m: \real \rightarrow \real,$ is a given continuous and
$2\pi-$periodic function such that for some $p_i \in [1,\infty], \
1 \leq i \leq n,$
\begin{equation}\label{2605101}
\begin{array}{c} m(t) \geq 0, \ \forall \ t
\in \real\ \mbox{and} \ m \ \mbox{is not identically zero.} \\
\Vert m \Vert_{p_i} < \frac{\beta_{p_i}^{per}}{b_i}, \ \mbox{if} \
p_i \in (1,\infty], \ \ \Vert m \Vert_{p_i} \leq
\frac{\beta_{p_i}^{per}}{b_i}, \ \mbox{if} \ p_i =1,
\end{array}
\end{equation}
then for each continuous and $2\pi-$ periodic function $h:\real
\rightarrow \real^n,$ the boundary value problem
\begin{equation}\label{3eq17}
u''(t)+ m(t)H_u (u(t)) = h(t),  \ t\in \real,  \ u(0)-u(2\pi) =
u'(0)-u'(2\pi) = 0,
\end{equation}
has a unique solution.

If in (\ref{2605101}) we choose $p_i \neq \infty,$ for some $1
\leq i \leq n,$ then it is clear that the eigenvalues of the
matrix $m(t)H_{uu}(u)$ in the boundary value problem
(\ref{3eq17}), can cross an arbitrary number of eigenvalues of
(\ref{eq7}).

\end{example}

\section{Stability for linear periodic systems}

In this section we present some new conditions which allow to
prove that a given periodic linear and conservative system is
stably bounded. More precisely, we consider systems of the type
\begin{equation}\label{eq35}
u''(t) + P(t)u(t)= 0, \ t \in \real,
\end{equation}

\noindent where from now on we assume that the matrix function $P(\cdot) \in
\Lambda$ ($\Lambda$ was defined in the introduction).

The system (\ref{eq35}) is said to be stably bounded
(\cite{krein}) if there exists $\varepsilon (P) \in \real ^+,$
such that all solutions of the system
\begin{equation}\label{2eq35}
u''(t) + Q(t)u(t)= 0, \ t \in \real,
\end{equation}
are bounded for all matrix function $Q(\cdot) \in \Lambda,$
satisfying
$$
\displaystyle \max_{1 \leq i,j \leq n} \ \intl \vert p_{ij}(t) -
q_{ij}(t) \vert \ dt < \varepsilon.
$$

In \cite{krein}, Krein proved that all solutions of the system
(\ref{eq35}) are stably bounded if $\lambda_1 >1,$ where
$\lambda_1$ is the smallest positive eigenvalue of the eigenvalue
problem
\begin{equation}\label{eq37}
u''(t) + \mu P(t)u(t)= 0, \ t \in \real, \ u(0)+u(T) = u'(0)+u'(T)
= 0.
\end{equation}

Moreover, the eigenvalue $\lambda_1$ has a variational
characterization given by
\begin{equation}\label{eq38}
\displaystyle \frac{1}{\lambda_1} = \max_{y \in G_T}  \ \intl
\langle P(t)y(t),y(t)\rangle \  dt,
\end{equation}
where
\begin{equation}\label{39}
G_T = \{ y \in H^1 (0,T): y(0)+y(T) = 0, \ \dis \sum_{i=1}^{n}
\intl (y_i ' (t))^2 \ dt = 1 \}.
\end{equation}

Based on these previous results, we can prove the following
theorem.

\begin{theorem}\label{t2}
Let $P(\cdot) \in \Lambda$ be such that there exist a diagonal
matrix $B(t)$ with continuous and $T-$periodic  entries
$b_{ii}(t),$ and $p_i \in [1,\infty], \ 1 \leq i \leq n,$
satisfying
\begin{equation}\label{eq40}
\begin{array}{c}
P(t) \leq B(t), \ \forall \ t \in \real, \\ \\ \Vert b_{ii}^+
\Vert_{p_i} < \beta_{p_i}^{ant}, \ \mbox{if} \ p_i \in (1,\infty],
\ \ \Vert b_{ii}^+ \Vert_{p_i} \leq  \beta_{p_i}^{ant}, \
\mbox{if} \ p_i =1.
\end{array}
\end{equation}
Then, the system (\ref{eq35}) is stably bounded.
\end{theorem}
\begin{proof}
Let $y \in G_T.$ Then by using the Lemma \ref{l2}, we have
\begin{equation}\label{eq41}
\begin{array}{c}
\intl \langle P(t)y(t),y(t)\rangle  \ dt \leq \intl \langle B(t)y(t),y(t)\rangle  \ dt \leq \\
\\
\dis \sum_{i=1}^{n} \ \dis \intl b_{ii}(t)(y_i (t))^2 (t) \ dt
\leq \dis \sum_{i=1}^{n} \Vert b_{ii}^+ (t) \Vert_{p_i} \Vert y_i
^2 \Vert_{\frac{p_i}{p_i -1}}\leq  \\ \\
\leq  \dis \sum_{i=1}^{n} \beta_{p_i}^{ant} \Vert y_i ^2
\Vert_{\frac{p_i}{p_i -1}} \leq \dis \sum_{i=1}^{n} \intl (y_i '
(t))^2 \ dt = 1, \ \forall \ y \in G_T
\end{array}
\end{equation}
where
$$
\begin{array}{c}
\frac{p_i}{p_i -1} = \infty, \ \ \mbox{if} \ p_i = 1 \\ \\
\frac{p_i}{p_i -1} = 1, \ \ \mbox{if} \  p_i = \infty. \end{array}
$$
At this point, we claim
\begin{equation}\label{eq42}
 \dis \frac{1}{\lambda_1} < 1.
 \end{equation}

In fact, (\ref{eq41}) implies $\dis \frac{1}{\lambda_1} \leq 1.$
Now, if $\lambda_1 = 1,$ let us choose $y(\cdot)$ as any
nontrivial eigenfunction associated to $\mu = 1$ in (\ref{eq37}),
i.e.,
\begin{equation}\label{3eq37}
y''(t) + P(t)y(t)= 0, \ t \in \real, \ y(0)+y(T) = y'(0)+y'(T) =
0.
\end{equation}
Then some component, say $y_j,$ is nontrivial. If $p_j \in
(1,\infty],$ then $(\beta_{p_j}^{ant}-\Vert b_{jj}^+
\Vert_{p_j})\Vert y_j^2 \Vert_{\frac{p_j}{p_j -1}} >0 $ and
$(\beta_{p_i}^{ant}-\Vert b_{ii}^+ \Vert_{p_i})\Vert y_i^2
\Vert_{\frac{p_i}{p_{i} -1}}\geq 0, \ \forall \ i \neq j$, so that
we have a strict inequality in (\ref{eq41}). This is a
contradiction with (\ref{3eq37}). If $p_j =1,$ we use the Remark
\ref{remark2405101}. Since $y_j \in C^1 [0,T],$ either $x_0 = 0$
or $x_0 = T.$ In any case the function $w$ of Remark
\ref{remark2405101} satisfies $w'(0) + w'(T) \neq 0.$ Then we have
$\beta_{p_j}^{ant} \Vert y_j ^2 \Vert_{\frac{p_j}{p_j -1}} < \intl
(y_j ' (t))^2 \ dt.$ In this case we have again a strict
inequality in (\ref{eq41}), which is a contradiction with
(\ref{3eq37}).
\end{proof}

\begin{remark}\label{remark3}
Previous Theorem is optimal in the following sense. For any given
positive numbers $\gamma_i, \ 1 \leq i \leq n,$ such that at least
one of them, say $\gamma_j,$ satisfies
\begin{equation}\label{22optimalidad}
\gamma_j > \beta_{p_j} ^{ant}, \ \mbox{for some} \ p_j \in
[1,\infty],
\end{equation} there exists a diagonal  $n\times n$ matrix
$P(\cdot) \in \Lambda$ with entries $p_{ii}(t), \ 1 \leq i \leq
n,$ satisfying $\Vert p_{ii}^+ \Vert_{p_i} < \gamma_{i}, \ 1 \leq
i \leq n$ and such that the system (\ref{eq35}) is not stable.

To see this, if $\gamma_j$ satisfies (\ref{22optimalidad}), then
there exists some continuous and $T-$periodic function $p(t), $
not identically zero, with $\intl p (t) \ dt \geq 0, $ and $\Vert
p^+ \Vert_{p_j} < \gamma_j,$ such that the scalar problem
$$w''(t) + p(t)w(t) = 0,
$$
is not stable (see Theorem 1 in \cite{zhangli}). If we choose
$$
p_{jj}(t) = p(t), \ \ p_{ii}(t) = \delta \in \real^+, \ \mbox{if}
\  i\neq j,
$$
with $\delta$ sufficiently small, then (\ref{eq35}) is unstable.
\end{remark}
\begin{remark}\label{remark4}
The property of stable boundedness for the solutions of systems
like (\ref{eq35}) have been considered in \cite{clhi}. The authors
assume $L^1$ restrictions on the spectral radius of some
appropriate matrices which are calculated by using the matrix
$P(t).$ It is easy to check that, even in the scalar case, these
conditions are independent from classical $L^1-$ Lyapunov
inequality and therefore, they are also independent from our
results in this paper.
\end{remark}

\begin{example}\label{ejemplo240510t} Next we show a two dimensional
system where we provide sufficient conditions, which may be
checked directly by using the elements $p_{ij}$ of the matrix
$P(t),$ to assure that all hypotheses of the previous Theorem are
fulfilled. The example is based on a similar one, shown by the
authors in \cite{cavijde}, in the study of Lyapunov inequalities
for elliptic systems.

Let the matrix $P(t)$ be given by
\begin{equation}\label{teq10}
P(t) = \left (
\begin{array}{cc} p_{11}(t) & p_{12}(t) \\
p_{12}(t) & p_{22}(t) \end{array} \right )
\end{equation}
where \afirm{[{\bf H1}]}{10cm} {
$$
\begin{array}{c} p_{ij} \in C_T(\real,\real), \ 1 \leq i,j \leq 2, \\ \\  p_{11}(t) \geq 0, \
p_{22}(t) \geq 0, \ \det \ P(t) \geq 0, \ \forall \ t \in \real,
\\ \\ \det \ P(t) \neq 0, \ \mbox{for some}\ t \in \real.
\end{array}
$$ }
$C_T (\real,\real)$ denotes the set of real, continuous and
$T-$periodic functions defined in $\real.$ In addition, let us
assume that there exist $p_1, p_2 \in (1,\infty]$ such that
\begin{equation}\label{teq11}
\Vert p_{11}\Vert_{p_1} < \beta_{p_1}^{ant}, \ \ \Vert p_{22} +
\displaystyle \frac{p_{12}^2}{\beta_{p_1}^{ant} - \Vert p_{11}
\Vert_{p_1}} \Vert _{p_2} < \beta_{p_2}^{ant}.
\end{equation}
Then (\ref{eq35}) is stably bounded.

In fact, it is trivial to see that [{\bf H1}] implies that the
eigenvalues of the matrix $P(t)$ are both nonnegative, which
implies that $P(t)$ is positive semi-definite. Also, since  $\det
\ P(t) \neq 0, \ \mbox{for some}\ t \in \real,$ (\ref{eq35}) has
not nontrivial constant solutions. Therefore, $P(\cdot) \in
\Lambda,$ the set defined in the Introduction. Moreover, it is
easy to check that for a given diagonal matrix $B(t),$ with
continuous entries $b_{ii}(t), \ 1 \leq i \leq 2,$ the relation
\begin{equation}\label{teq2306084}
P(t) \leq B(t), \ \forall \ t \in \real
\end{equation}
is satisfied if and only if, $\forall \ t \in \real$ we have
\begin{equation}\label{teq12}
\begin{array}{c}
\ b_{11}(t) \geq p_{11}(t), \ b_{22}(t) \geq  p_{22}(t),  \\
(b_{11}(t)- p_{11}(t))(b_{22}(t)- p_{22}(t)) \geq p_{12}^2 (t).
\end{array}
\end{equation}
In our case, if we choose
\begin{equation}\label{teq13}
b_{11}(t) = p_{11}(t) + \gamma, \ b_{22}(t) = p_{22}(t) +
\displaystyle \frac{p_{12}^2 (t)}{\gamma}
\end{equation}
where $\gamma$ is any constant such that
\begin{equation}
\begin{array}{c}
0 < \gamma < \beta_{p_1}^{ant} - \Vert p_{11} \Vert_{p_1}, \\
\left ( \frac{1}{\gamma} - \frac{1}{\beta_{p_1} - \Vert p_{11}
\Vert _{p_1}} \right ) \Vert p_{12}^2 \Vert_{p_2} <
\beta_{p_2}^{ant} - \Vert p_{22} + \displaystyle
\frac{p_{12}^2}{\beta_{p_1}^{ant} - \Vert p_{11} \Vert_{p_1}}
\Vert _{p_2},
\end{array}
\end{equation}
then all conditions of Theorem \ref{t2} are fulfilled and
consequently (\ref{eq35}) is stably bounded.
\end{example}
\begin{remark}
Let us observe that we deduce from (\ref{teq11})
\begin{equation}\label{2teq11}
\Vert p_{11}\Vert_{p_1} < \beta_{p_1}^{ant}, \ \ \Vert p_{22}
\Vert _{p_2} < \beta_{p_2}^{ant}.
\end{equation}
As a consequence, the uncoupled system
\begin{equation}\label{14092010}
v''(t) + R(t)v(t) = 0, \ t \in \real,
\end{equation}
where
\begin{equation}\label{2eq10}
R(t) = \left (
\begin{array}{cc} p_{11}(t) & 0 \\
0 & p_{22}(t) \end{array} \right )
\end{equation}
is stably bounded.

Therefore, by using the definition of stably bounded system,
(\ref{eq35}) is stably bounded for any continuous and $T-$periodic
function $p_{12}$ with sufficiently small $L_{1}-$ norm. However,
(\ref{teq11}) does not imply, necessarily, that the $L_{1}-$norm
of the function $p_{12}$ is necessarily small.
\end{remark}

\end{document}